\documentclass[10pt]{article}
\usepackage{booktabs}
\usepackage{caption}
\usepackage{amsmath}
\usepackage{amsfonts,amsthm,amssymb,mathrsfs,bbding}
\usepackage{enumerate}
\usepackage{tikz}
\usepackage{rotating}
\usepackage{tabularx}
\usepackage{pifont}
\allowdisplaybreaks[4]
\evensidemargin=\oddsidemargin\topmargin=-1.5cm

\newtheorem{thm}{Theorem}[section]
\newtheorem{prob}{Problem}

\newtheorem{lem}[thm]{Lemma}

\newtheorem{conj}{Conjecture}

\theoremstyle{definition}

\begin{document}
\title{\bf A note on the spectral radius and $[a,b]$-factor of graphs\footnote{This work was supported by the National Natural Science Foundation of China (Grant Nos. 12271162,
12301454), Natural Science Foundation of Shanghai (Nos. 22ZR1416300 and 23JC1401500) and the Program for Professor of Special Appointment (Eastern Scholar) at Shanghai Institutions of Higher Learning (No. TP2022031),}\setcounter{footnote}{-1}\footnote{\emph{Email address:} ddfan0526@163.com (D. Fan), huiqiulin@126.com (H. Lin), zyf1218@sina.cn (Y. Zhu).}}
\author{Dandan Fan$^{a}$,  Huiqiu Lin$^b$\thanks{Corresponding author.},  Yinfen Zhu$^{b,c}$\\[2mm]
\small\it $^a$ School of Mathematics and Physics, Xinjiang Agricultural University\\
\small\it Urumqi, Xinjiang 830052, China\\[1mm]
\small\it $^b$ School of Mathematics, East China University of Science and Technology, \\
\small\it   Shanghai 200237, China \\[1mm]
\small\it $^c$ School of Mathematics and Physics, Xinjiang Institute of Engineering,\\
\small\it Urumqi, Xinjiang 830023, China}
\date{}
\maketitle
{\flushleft\large\bf Abstract}
The investigation of eigenvalue conditions for the existence of an $[a,b]$-factor originates in the work of Brouwer and Haemers (2005) on perfect matchings.
In the decades since, spectral extremal problems related to $[a,b]$-factors have attracted considerable attention.
In this paper, we establish a spectral radius condition that ensures the existence of an $[a,b]$-factor in a graph $G$ with minimum degree $\delta(G) \geq a$, where $b > a \geq 1$.
This result resolves a problem posed by Hao and Li [Electron. J. Combin. (2024)].

	\begin{flushleft}
	\end{flushleft}
	
	\section{Introduction}	
 %and extended in \cite{H.L-3,H.L-4} to obtain a regular factor
An \textit{$[a,b]$-factor} of a graph $G$ is a spanning subgraph $H$ such that $a\leq d_{H}(v)\leq b$ for each $v\in V(G)$. The study of factors from graph eigenvalues has a rich history. A 1-factor is a perfect matching of graphs.
In 2005, Brouwer and Haemers \cite{A.B} described a regular graph to contain a perfect matching in terms of the third largest adjacency eigenvalue, which was improved in \cite{S.C,S.C-1,S.C-2}. Subsequently, O \cite{S.O} extended the result of Brouwer and Haemers to graphs that are not necessarily regular, and provided a spectral condition to guarantee the existence of a perfect matching in a graph. A $[k,k]$-factor is a $k$-factor in graphs. In 1891, Petersen \cite{Petersen1891} demonstrated that every $2r$-regular graph possesses a $2k$-factor for $0 \leq k \leq r$. One might naturally ask what the spectral condition is that guarantees the existence of a
$k$-factor in a general graph. Based on the Tutte's $k$-Factor Theorem (see \cite{Tutte1952}), Lu \cite{Lu2012} gave an answer to this question in terms of the third largest adjacency eigenvalue. By considering the condition of edge-connectedness, Gu \cite{Gu2014} studied eigenvalue conditions of $t$-edge-connected regular graphs with $k$-factors. Up to now, the problem of spectral conditions for the existence of an $[a,b]$-factor when $a$ and $b$ take special values has been well studied, such as, odd $[1,b]$-factors \cite{Kim-O-Park-Ree, LM22, LWY10}, $[1,b]$-factors \cite{Fan-Lin-Lu} and $[a,b]$-parity factors \cite{FLA24,Kim-O2023,O2022}. For more results, we refer the readers to \cite {FLA25, FL24, HLY25}.

For a graph $G$, let $A(G)$ denote  the \textit{adjacency matrix} of $G$. The largest eigenvalue of $A(G)$, denoted by $\rho(G)$, is called the \textit{spectral radius} of $G$. Let $G_{1} \vee G_{2}$ be the graph obtained from the disjoint union $G_{1}\cup G_{2}$ by adding all edges between $G_{1}$ and $G_{2}$.
Among graphs of order $n$ without an $[a, b]$-factor, Cho, Hyun, O and Park \cite{E.C} guessed that the graph $K_{a-1}\vee(K_{1}\cup K_{n-a})$ attains the maximum spectral radius, and they proposed the following conjecture.

\begin{conj}[See \cite{E.C}]\label{conj::1.1}
Let $a\cdot n$ be an even integer at least 2, where $n\geq a+1$. If $G$ is a graph of order $n$ with $$\rho(G)>\rho(K_{a-1}\vee(K_{1}\cup K_{n-a})),$$
 then $G$ contains an $[a,b]$-factor.
\end{conj}

Around this conjecture, Fan, Lin and Lu \cite{Fan-Lin-Lu} provided a partial result for $n\geq 3a+b-1$. The conjecture was then fully resolved by Wei and Zhang \cite{Wei-Zhang}. Recently, Hao and Li \cite{HL24} showed the result in this conjecture to be essentially best possible by constructing extremal graphs, and they also studied a similar problem in bipartite graphs.
Observe that $\delta(G)\geq a$ is a necessary condition for a graph $G$ to contain an $[a,b]$-factor. Hao and Li \cite{HL24} put forward the following problem.
\begin{prob}[Problem 35, \cite{HL24}]\label{pro-1}
Determine sharp lower bounds on the size or spectral radius of an $n$-vertex graph $G$ with minimum degree $\delta(G) \geq a$ such that $G$ contains an $[a,b]$-factor.
\end{prob}

Very recently, Tang and Zhang \cite{TZ2025} gave the answer to Problem 1 from the perspective of spectral radius for $a=b$. In this paper, we extended their result by investigating the remaining case of $b>a\geq 1$. Suppose that $H_{n}^{a,b}$ is a graph obtained from $K_a\vee (K_{n-a-b-1}\cup (b+1)K_1)$ by adding $a-1$ edges between one vertex in $(b+1)K_1$ and $a-1$ vertices in $K_{n-a-b-1}$.

\begin{thm}\label{thm::1.1}
		Let $a$ and $b$ be two positive integers with $b>a$, and let $G$ be a connected graph of order $n \geq 2(a+b+2)(b+2)$ with minimum degree $\delta(G)\geq a$. If
		$$\rho(G) \geq \rho(H_{n}^{a,b}),$$
		then $G$ contains an  $[a,b]$-factor, unless $G\cong H_{n}^{a,b}$.
	\end{thm}

	%From the size perspective, we also consider a size condition to guarantee the existence of a fractional $[a,b]$-factor in a graph. Denote by $e(G)$ the number of edges in $G$.
%	\begin{thm}\label{thm::1.4}
%		Let $b\geq a \geq 1$ be two integers, and let $G$ be a connected graph of order $n\geq 4a+\frac{5b}{2}+6$ with minimum degree $\delta(G)\geq a$. If
%		$$e(G) \geq {n-b-1\choose 2}+ab+2a,$$
%		then $G$ has a fractional $[a,b]$-factor.	
%	\end{thm}
%	The condition in Theorem~\ref{thm::1.4} is tight. Let $H_n^{a,b}$ be the graph obtained from $K_a \vee \left(K_{n-(a+b+1)} \cup (b+1) K_1\right)$ by adding $a-1$ edges between $(b+1) K_1$ and $K_{n-(a+b+1)}$. Observe that $e(H_n^{a,b})={n-b-1\choose 2}+ab+2a-1$, but $H_n^{a,b}$ contains no fractional $[a,b]$-factors.
%	

	\section{Proofs of Theorem~\ref{thm::1.1}}
In this section, we give the proof of Theorem \ref{thm::1.1}. The following lemmas are used in the sequel. For any $v\in V(G)$, let $N_{G}(v)$ be the neighborhood of $v$. Particularly, let $N_{G}[u]=N_{G}(u)\cup\{u\}$.
\begin{lem}[See \cite{H.L-1}]\label{lem::2.0}
Let $G$ be a connected graph, and let $u,v$ be two vertices of $G$. Suppose that $v_{1},v_{2},\ldots,v_{s}\in N_{G}(v)\backslash N_{G}(u)$ with $s\geq 1$, and $G^*$ is the graph obtained from $G$ by deleting the edges $vv_{i}$ and adding the edges $uv_{i}$ for $1\leq i\leq s$. Let $x$ be the Perron vector of $A(G)$. If $x(u)\geq x(v)$, then $\rho(G)<\rho(G^*)$.
\end{lem}
The following sharp upper bound on the spectral radius was obtained by Hong, Shu and Fang ~\cite{HSF} and Nikiforov~\cite{V.N}, independently.		
	\begin{lem}[See \cite{HSF,V.N}]\label{lem::2.1}
		Let $G$ be a graph on $n$ vertices and $m$ edges with minimum degree $\delta\geq 1$. Then$$\rho(G) \leq \frac{\delta-1}{2}+\sqrt{2 m-n \delta+\frac{(\delta+1)^{2}}{4}},$$
		with equality if and only if $G$ is either a $\delta$-regular graph or a bidegreed graph
		in which each vertex is of degree either $\delta$ or $n-1$.
	\end{lem}

	\begin{lem}[See \cite{HSF,V.N}]\label{lem::2.2}
		For nonnegative integers $p$ and $q$ with $2q \leq p(p-1)$ and $0 \leq x \leq p-1$, the function $f(x)=(x-1) / 2+\sqrt{2 q-p x+(1+x)^{2} / 4}$ is decreasing with respect to $x$.
	\end{lem}
	
For any $S\subseteq V(G)$, let $G[S]$ be the subgraph of $G$ induced by $S$, and
	let $e(S)$ be the number of edges in $G[S]$. For any vertex $v\in V(G)$ and subset $S\subseteq V(G)$, suppose that $N_{S}(v)=N_{G}(v)\cap S$ and $d_{S}(v)=|N_{S}(v)|$.  Particularly, let $G-S=G[V(G)\backslash S]$.
Another version of Tutte $[a,b]$-factor Theorem (see \cite{Tutte1952}) is observed by Heinrich et al. in 1990, which uses only one subset $S$ of $V(G)$ instead of two subsets $S$ and $T$ of $V(G)$.
	
	\begin{lem}[See \cite{HHK90}]	\label{lem::2.3}
		Let $a$ and $b$ be two positive integers with $b>a$. Then a graph $G$ contains an $[a,b]$-factor if and only if for any $S \subseteq V(G)$,
		$$
		a|W|-\sum_{v\in W}d_{G-S}(v) \leq b|S|,
		$$
where $W=\{v \in V(G)\backslash S|~d_{G-S}(v)<a\}$.
	\end{lem}

\begin{lem}[See \cite{H.L-1}]\label{lem::adjacent}
Let $u, v$ be two distinct vertices of a connected graph $G$, and let $x$ be the Perron vector of $A(G)$.

\noindent (i) If $N_G(v)\backslash\{u\}\subset N_G(u)\backslash\{v\}$, then $x(u)> x(v)$.

\noindent (ii) If $N_G(v)\subseteq N_G[u]$ and $N_G(u)\subseteq N_G[v]$, then $x(u)=x(v)$.
\end{lem}

Denote by $\mathcal{G}_{n}^{a,b}$ the set of graphs obtained from $K_{a}\vee (K_{n-a-b-1}\cup (b+1)K_{1})$ by adding $a-1$ edges between $V(K_{n-a-b-1})$ and $V((b+1)K_{1})$.

\begin{lem}\label{lem::n-b-1}
Let $a$ and $b$ be two positive integers with $b>a$. If $G\in \mathcal{G}_{n}^{a,b}$ and $n\geq 2(a+b+2)(b+2)$, then
$$n-b-2< \rho(G)< n-b-1.$$
\end{lem}

\begin{proof}
Note that $G$ contains $K_{n-b-1}\cup (b+1)K_{1}$ as a proper spanning subgraph. Then $\rho(G)>\rho(K_{n-b-1}\cup (b+1)K_{1})= n-b-2$. Since $G\in \mathcal{G}_{n}^{a,b}$, it follows that
\begin{equation*}
\begin{aligned}
   e(G)={n-b-1\choose 2}+a(b+1)+a-1.
\end{aligned}
\end{equation*}
Combining this with $n\geq 2(a+b+2)(b+2)$, Lemmas \ref{lem::2.1} and \ref{lem::2.2}, we have
\begin{equation*}
\begin{aligned}
   \rho(G)&\leq \frac{a-1}{2}+\sqrt{(n-b-1)(n-b-2)+2a(b+1)+2a-2-an+\frac{(a+1)^2}{4}}\\
   &= \frac{a-1}{2}+\sqrt{\left(n-b-\frac{a}{2}-\frac{1}{2}\right)^2-(2n-ab-4a-2b)}\\
   &< \frac{a-1}{2}+ \left(n-b-\frac{a}{2}-\frac{1}{2}\right)~~(\mbox{since $n\geq 2(a+b+2)(b+2)$ and $b>a\geq 1$})\\
   &=n-b-1.
\end{aligned}
\end{equation*}
This completes the proof.
\end{proof}

\begin{lem}\label{lem::maximum}
Let $a$ and $b$ be two positive integers with $b>a$. If $G\in \mathcal{G}_{n}^{a,b}$ and $n\geq 2(a+b+2)(b+2)$, then
$$\rho(G)\leq \rho(H_{n}^{a,b}),$$
with equality if and only if $G\cong H_{n}^{a,b}$.
\end{lem}

\begin{proof}
Suppose that $G$ is the graph that attains the maximum spectral radius in $\mathcal{G}_{n}^{a,b}$.
We partition $V(G)=U\cup S\cup W$ with $U=V(K_{n-a-b-1})=\{v_1,v_2,\ldots,v_{n-a-b-1}\}$, $S=V(K_a)=\{u_{1},u_2,\ldots,u_a\}$ and $W=V(G)\backslash (S\cup U)=\{w_1,w_2,\ldots,w_{b+1}\}$. Let $x$ be the Perron vector of $A(G)$, and let $\rho=\rho(G)$. Without loss of generality, we assume that $x(v_{i})\geq x(v_{i+1})$ and $x(w_{j})\geq x(w_{j+1})$ for $1\leq i\leq n-a-b-2$ and $1\leq j\leq b$. We assert that $N_{G}(v_{i+1})\subseteq N_{G}[v_i]$ and $N_{G}(w_{j+1})\subseteq N_{G}[w_j]$ for $1\leq i\leq n-a-b-2$ and $1\leq j\leq b$. Otherwise, there exist $i,j$ with $i<j$ such that $N_{G}(v_j)\nsubseteq N_{G}[v_i]$. Let
$w\in N_{G}(v_j)\backslash N_{G}[v_i]$ and let $G^{*}=G-wv_{j}+wv_{i}$. Then $G^*\in \mathcal{G}_{n}^{a,b}$. Since $x(v_{i})\geq x(v_{j})$, we have $\rho(G^*)>\rho(G)$ by Lemma \ref{lem::2.0}, which contradicts the maximality of $\rho$. This implies that  $N_{G}(v_{i+1})\subseteq N_{G}[v_i]$ for $1\leq i\leq n-a-b-2$. Similarly, we can deduce that $N_{G}(w_{j+1})\subseteq N_{G}[w_j]$ for $1\leq j\leq b$.
%Furthermore, we have
%$N_{V_2}(u_{i+1})\subseteq N_{V_2}(u_i)$ and $N_{V_1}(v_{j+1})\subseteq N_{V_1}(v_j)$ for $1\leq i\leq \delta$ and $1\leq j\leq n-\delta-2$.
Let $d_{U}(w_1)=p_1$, $d_{U}(w_2)=p_2$ and $d_{W}(v_1)=q$. Again by the maximality of $\rho$ and Lemma \ref{lem::2.0}, we have $N_{U}(w_1)=\{v_1,v_2,\dots,v_{p_1}\}$, $N_{U}(w_2)=\{v_1,v_2,\dots,v_{p_2}\}$ and $N_{W}(v_1)=\{w_1,w_2,\dots,w_q\}$. If $p_1=a-1$ or $q=1$, then $G\cong H_{n}^{a,b}$, as required. Next we consider $p_1\leq a-2$ and $q\geq 2$ in the following.
Note that $x(v_i)=x(v_{{p_1}+1})$ for $p_1+2\leq i\leq n-a-b-1$, $x(u_j)=x(u_{1})$ for $2\leq j\leq a$. Then, by $A(G)x=\rho x$, we have
\begin{align}
  \rho x(v_1)&=\sum_{2\leq i\leq p_1}x(v_i)\!+\!(n\!-\!a\!-\!b\!-\!p_1\!-\!1)x(v_{p_1+1})\!+\!ax(u_1)\!+\!\sum_{1\leq i\leq q}x(w_i)\label{equ::inset-1}\\
 \rho x(v_{p_{1}+1})&=x(v_1)\!+\!\sum_{2\leq i\leq p_1}x(v_i)\!+\!(n\!-\!a\!-\!b\!-\!p_1\!-\!2)x(v_{p_1+1})\!+\!ax(u_1),\label{equ::inset-2}\\
\rho x(w_1)&=x(v_1)\!+\!\sum_{2\leq i\leq p_1}x(v_i)\!+\!ax(u_1)\label{equ::inset-3}.
  \end{align}
Thus, by (\ref{equ::inset-2}) and (\ref{equ::inset-3}), we get
\begin{eqnarray}\label{equ::max-2}
 x(w_1)=\frac{\rho-(n-a-b-p_1-2)}{\rho}x(v_{p_1+1}).
\end{eqnarray}
Combining this with (\ref{equ::inset-1}), (\ref{equ::inset-2}), Lemma \ref{lem::n-b-1} and $x(w_i)\leq x(w_1)$ for $2\leq i\leq q$, we get
\begin{equation*}
\begin{aligned}
(\rho+1)x(v_1)&=(\rho+1)x(v_{p_1+1})+\sum_{1\leq i\leq q}x(w_i)\\
   &\leq(\rho+1)x(v_{p_1+1})+qx(w_1)~~(\mbox{since $x(w_1)\geq x(w_i)$ for $2\leq i\leq q$})\\
   &=\frac{\rho^2+\rho+q(\rho-(n-a-b-p_1-2))}{\rho}x(v_{p_1+1}) ~~(\mbox{by (\ref{equ::max-2})})\\
   &<\frac{\rho^2+\rho+q(p_1+a+1)}{\rho}x(v_{p_1+1})~~(\mbox{since $\rho<n-b-1$}),\\
\end{aligned}
\end{equation*}
and hence
\begin{eqnarray}\label{equ::max-3}
 x(v_{p_1+1})>\frac{\rho(\rho+1)}{\rho^2+\rho+q(p_1+a+1)}x(v_1).
\end{eqnarray}

Assume that $E_1=\{w_1v_{i}|~p_1+1\leq i\leq a-1\}$ and $E_2=\{w_{i}v_j\in E(G) |~ 2\leq i\leq q, 1\leq j\leq p_2\}$. Let $G'=G-E_2+E_1$. Clearly, $G'\cong H_{n}^{a,b}$. Let $y$ be the Perron vector of $A(G')$, and let $\rho'=\rho(G')$. By symmetry, we get $y(v_i)=y(v_1)$ for $2\leq i\leq a-1$, $y(v_i)=y(v_a)$ for $a+1\leq i\leq n-a-b-1$, $y(u_i)=y(u_1)$ for $2\leq i\leq a$ and $y(w_i)=y(w_2)$ for $3\leq i\leq b+1$. From $A(G')y=\rho'y$, we have
\begin{align}
  \rho' y(w_1)&=(a-1)y(v_1)+ay(u_1)\label{equ::max-4}\\
 \rho' y(w_2)&=ay(u_1),\label{equ::max-5}\\
\rho' y(v_1)&=(a-2)y(v_1)+(n-2a-b)y(v_a)+ay(u_1)+y(w_1)\label{equ::max-6},\\
\rho' y(v_a)&=(a-1)y(v_1)+(n-2a-b-1)y(v_a)+ay(u_1)\label{equ::max-7},
\end{align}
Putting (\ref{equ::max-4}) into  (\ref{equ::max-7}), we get
\begin{eqnarray*}
 y(v_a)=\frac{\rho'}{\rho'-(n-2a-b-1)}y(w_1).
\end{eqnarray*}
Combining this with (\ref{equ::max-5}) and (\ref{equ::max-6}), we obtain that
\begin{eqnarray*}
(\rho'-a+2)y(v_1)=\frac{(n-2a-b)\rho'}{\rho'-(n-2a-b-1)}y(w_1)+y(w_1)+\rho'y(w_2),
\end{eqnarray*}
which leads to
\begin{eqnarray}\label{equ::max-8}
y(v_1)=\frac{(n-2a-b+1)\rho'-(n-2a-b-1)}{(\rho'-(n-2a-b-1))(\rho'-a+2)}y(w_1)+\frac{\rho'}{\rho'-a+2}y(w_2).
\end{eqnarray}
Substituting (\ref{equ::max-5}) and (\ref{equ::max-8}) into (\ref{equ::max-4}), we have
\begin{eqnarray}\label{equ::max-9}
 y(w_1)=\frac{(\rho'^2+\rho')(\rho'-(n-2a-b-1))}{\rho'^3\!-\!(n\!-\!a\!-\!b\!-\!3)\rho'^2-(n-b-3)\rho'+(a-1)(n-2a-b-1)}y(w_2).
\end{eqnarray}
Note that $x(w_1)\geq x(w_i)$ for $2\leq i\leq b+1$, $x(v_1)\geq x(v_j)$ for $2\leq j\leq p_1$. Combining this with (\ref{equ::max-3}) and (\ref{equ::max-9}), we have
\begin{equation*}
\begin{aligned}
&~~~y^{T}(\rho'-\rho)x\\
&=y^{T}(A(G')-A(G))x\\
   &=\sum_{w_1v_i\in E_1}(x(w_{1})y(v_{i})\!+\!x(v_{i})y(w_{1}))\!
-\sum_{w_iv_j\in E_2}\!(x(w_{i})y(v_{j})\!+\!x(v_{j})y(w_{i}))\\
   &\geq(a\!-\!1\!-\!p_1)(x(w_1)y(v_1)\!+\!x(v_{p_1+1})y(w_1)\!-\!x(w_2)y(v_1)\!-\!x(v_1)y(w_2))~~(\mbox{since $p_1\leq a\!-\!2$})\\
   &\geq (a\!-\!1-p_1)(x(v_{p_1+1})y(w_1)-x(v_1)y(w_2))~~~(\mbox{since $x(w_1)\geq x(w_2)$})\\
   &>\Big(\frac{\rho(\rho+1)}{\rho^2+\rho+q(p_1+a+1)}\cdot
   \frac{(\rho'^2+\rho')(\rho'\!-\!(n\!-\!2a\!-\!b\!-\!1))}{\rho'^3\!-\!(n\!-\!a\!-\!b\!-\!3)\rho'^2-(n-b-3)\rho'\!+\!(a\!-\!1)(n\!-\!2a\!-\!b\!-\!1)}\!-\!1\Big)\cdot\\
   &(a\!-\!1-p_1)x(v_{1})y(w_2)~~~(\mbox{by (\ref{equ::max-3}) and (\ref{equ::max-9})})\\
   &=\frac{(a\!-\!1-p_1)x(v_{1})y(w_2)\cdot f(\rho,\rho')}{(\rho^2+\rho+q(p_1+a+1))(\rho'^3\!-\!(n\!-\!b\!-\!3)(\rho'^2\!+\!\rho)
   \!+\!a\rho'^2\!+\!(a\!-\!1)(n\!-\!2a\!-\!b\!-\!1))},
\end{aligned}
\end{equation*}
where
\begin{equation*}
\begin{aligned}
f(\rho,\rho')=&\rho(\rho+1)(\rho'^2+\rho')(\rho'-(n-2a-b-1))-((\rho^2+\rho+p(p_1+a+1))\cdot\\
&(\rho'^3\!-\!(n\!-\!b\!-\!3)(\rho'^2\!+\!\rho)\!+\!a\rho'^2\!+\!(a\!-\!1)(n\!-\!2a\!-\!b\!-\!1)))\\
=&(a\!-\!1)(\rho^2\rho'^2\!+\!2\rho^2\rho'\!+\!\rho\rho'^2\!+\!2\rho\rho'\!-\!(n\!-\!2a\!-\!b\!-\!1)(\rho^2\!+\!\rho))\!-\!q(p_1\!+\!a\!+\!1)\rho'^3
\!+\!\\
&q(p_1\!+\!a\!+\!1)((n-a-b-3)\rho'^2+(n-b-3)\rho'-(a-1)(n-2a-b-1))\\
\geq&(a\!-\!1)[(\rho^2\!-\!(p_1\!+\!a\!+\!1)\rho')\rho'^2\!+\!(2\rho'\!-\!(n\!-\!2a\!-\!b\!-\!1))(\rho^2\!+\!\rho)]\\
&(\mbox{since $2\leq q\leq a\!-\!1$, $a\geq 1$, $n\geq 2(a\!+\!b\!+\!2)(b\!+\!2)$, $\rho'>n\!-\!b\!-\!2$ and $p_1\geq 0$})\\
\geq&(a\!-\!1)[((n\!-\!b\!-\!2)^2\!-\!(p_1\!+\!a\!+\!1)(n\!-\!b\!-\!1))\rho'^2\!+\!(2(n\!-\!b\!-\!2)\!-\!(n\!-\!2a\!-\!b\!-\!1))\cdot\\
&(\rho^2\!+\!\rho)]~~~(\mbox{since $\rho> n-b-2$, $p_1\leq a-2$ and $n-b-2<\rho'<n-b-1$})\\
\geq&0 ~~~(\mbox{since $a\geq 1$, $p_1\leq a-2$, $b>a$ and $n\geq 2(a+b+2)(b+2)$}).
\end{aligned}
\end{equation*}
It follows that $f(\rho,\rho')\geq 0$, and hence $\rho'>\rho$, which contradicts the maximality of $\rho$.

This completes the proof.\end{proof}

%For $X,Y\subseteq V(G)$, we denote by $E_{G}(X,Y)$ the set of edges with one endpoint in $X$ and one endpoint in $Y$, and $e_{G}(X,Y)=|E_{G}(X,Y)|$.
%\begin{lem}\label{lem::2.6}
%Let $k\geq 2$, and let $G\in \mathcal{G}_{n,a}^{k-1}$ where $n\geq 2a$, $a\geq \delta+2$ and $\delta\geq 2k$. Then $\rho(G)<\rho(B_{n,\delta+1}^{k-1}).$
%\end{lem}

 For every two disjoint subsets $A$ and $B$ of $V(G)$,
	let $E(A,B)$ denote the set of edges in $G$ with one end in $A$ and the other one in $B$, and let $e(A,B)=|E(A,B)|$. Now, we shall give the proof of Theorem \ref{thm::1.1}.
	
	\begin{proof}[\bf Proof of Theorem \ref{thm::1.1}]
Suppose that $G$ is a graph with the maximum spectral radius contains no $[a,b]$-factors, where $b>a\geq 1$. By Lemma \ref{lem::2.3}, there exists some $S \subseteq V(G)$ satisfying $|S|$ as large as possible such that $a|W|-\sum_{v \in W} d_{G-S}(v) \geq b|S|+1$, where $W=\{v \in V(G)\backslash S~|~d_{G-S}(v)<a\}$. Let $|W|=t$ and $|S|=s$. Then
		\begin{eqnarray}\label{equ::thm1-1}
			\sum_{v \in W} d_{G-S}(v) \leq at-b s-1.
		\end{eqnarray}
 We have the following three claims.
		
		%{\flushleft\bf{Claim 1.}} $s\geq 1$.
%		
%		Otherwise, $s=0$. By (\ref{equ::thm1-1}) and $\delta(G) \geq a$, we have
%		
%		\begin{equation*}
%			\begin{aligned}
%				a t \leq \sum_{v \in T} d_{G}(v)=\sum_{v \in T} d_{G-S}(v)\leq a t-1,
%			\end{aligned}
%		\end{equation*}
%		a contradiction. It follows that $s\geq 1$, proving Claim 1. \qed

		{\flushleft\bf{Claim 1.}} $t \geq b+1$.
		
		Note that $\delta(G)\geq a$ and $d_{G-S}(v)\geq \delta(G)-s\geq a-s$ for any $v\in W$. Then
		$$(a-s) t \leq \sum_{v \in W} d_{G-S}(v) \leq at-b s-1 $$
		due to (\ref{equ::thm1-1}), and hence $t \geq b+1/s.$ It follows that $t \geq b+1$ since $t$ is a positive integer and $s\geq 1$, as required. \qed

		{\flushleft\bf{Claim 2.}} $s \leq t-1$.
		
		Otherwise, $s \geq t$. By (\ref{equ::thm1-1}), we have
		\begin{equation*}
			\begin{aligned}
				0\leq \sum_{v \in W} d_{G-S}(v) \leq at-bs-1< -1,
			\end{aligned}
		\end{equation*}
		a contradiction. Thus, $s \leq t-1$, as required.\qed

{\flushleft\bf{Claim 3.}} $s\geq 1$.
		
Otherwise, $s=0$. By (\ref{equ::thm1-1}) and $\delta(G) \geq a$, we have
		
		\begin{equation*}
			\begin{aligned}
				a t \leq \sum_{v \in W} d_{G}(v)=\sum_{v \in W} d_{G-S}(v)\leq a t-1,
			\end{aligned}
		\end{equation*}
		a contradiction. It follows that $s\geq 1$.\qed

Note that $H_{n}^{a,b}$ contains no $[a,b]$-factors and $K_{n-b-1}$ is a proper subgraph of $H_{n}^{a,b}$. Combining this with the maximality of $\rho(G)$, we have
\begin{equation}\label{equ::lower}
			\begin{aligned}
				\rho(G)\geq \rho(H_{n}^{a,b})>\rho(K_{n-b-1})=n-b-2.
			\end{aligned}
		\end{equation}
Again by the maximality of $\rho(G)$, we can deduce that $G[V(G)\backslash W]\cong K_{n-t}$ and $e(S,W)=st$. By (\ref{equ::thm1-1}), we have
		\begin{equation}\label{equ::upper-edge}
			\begin{aligned}
				e(G)& \leq \sum_{v\in W} d_{G-S}(v)+e(S, W)+e(G[V(G)\backslash W]) \\
				& \leq a t-b s-1+s t+{n-t\choose 2}\\
				&=a t+s(t-b)-1+\frac{(n-t)(n-t-1)}{2}. \\
			\end{aligned}
		\end{equation}
We consider the following two cases depend on the value of $t$.

{\flushleft\bf{Case 1.}} $t\geq \frac{n}{b+2}$.

Since $n\geq 2(a+b+2)(b+2)$, we have $t\geq \frac{n}{b+2}\geq 2(a+b+2)$. Combining this with  (\ref{equ::upper-edge}), Claim 2, $\delta(G)\geq a$, Lemmas \ref{lem::2.1} and \ref{lem::2.2}, we get
		\begin{equation*}
			\begin{aligned}
				\rho(G) & \leq  \frac{a-1}{2}+\sqrt{2e(G)-n a+\frac{(a+1)^2}{4}} \\
				& \leq \frac{a\!-\!1}{2}\!+\!\sqrt{\Big(n\!-\!b\!-\!\frac{a}{2}\!-\!\frac{3}{2}\Big)^2\!-\!(2(t\!-\!b\!-\!1)n\!+\! a(b\!-\!2t\!+\!1)\!+\!b(b\!+\!2s\!+\!3)\!-\!t(2s\!+\!t\!+\!1)\!+\!4)}\\
				&\leq \frac{a\!-\!1}{2}\!+\!\sqrt{\Big(n\!-\!b\!-\!\frac{a}{2}\!-\!\frac{3}{2}\Big)^2\!-\!(t^2\!-\!(2b\!+\!2a\!+\!3)t\!+\!ab\!+\!b^2\!+\!a\!+\!3b-2s+4)}~~~~(\mbox{since $n\geq s+t$}) \\			
& \leq  \frac{a\!-\!1}{2}\!+\!\sqrt{\Big(n\!-\!b\!-\!\frac{a}{2}\!-\!\frac{3}{2}\Big)^2\!-\!(t^2\!-\!(2b\!+\!2a\!+\!5)t\!+\! ab\!+\!b^2\!+\!a\!+\!3b\!+\!6)}~~~~(\mbox{since $s\leq t-1$}) \\
& \leq  \frac{a\!-\!1}{2}\!+\!\sqrt{\Big(n\!-\!b\!-\!\frac{a}{2}\!-\!\frac{3}{2}\Big)^2\!-\!(ab+b^2-a+b+2)}~~~~(\mbox{since $t\geq 2(a+b+2)$}) \\
&<n-b-2 ~~~~(\mbox{since $b>a\geq 1$}),
			\end{aligned}
		\end{equation*}
which contradicts (\ref{equ::lower}).

{\flushleft\bf{Case 2.}} $b+1\leq t<\frac{n}{b+2}$.

In this case, we have $n\geq t(b+2)+1$. Let $x$ be the Perron vector of $A(G)$, and let
 $U=V(G)\backslash(S\cup W)=\{u_1,u_2,\ldots,u_{n-s-t}\}$ with $x(u_1)\geq x(u_2)\geq \cdots\geq x(u_{n-s-t})$. Then $l=|U|=n-s-t\geq bt+2$ due to $n\geq t(b+2)+1$ and $s\leq t-1$. We first assert that $e(W)=0$. Otherwise, $uv\in E(W)$. Observe that $d_{U}(u)\leq \sum_{v\in W}d_{G-S}(v)\leq at-bs-1<bt-1<l$. Thus, there exists some vertex $w\in U$ such that $vw\notin E(G)$. Assume that $z\in W$ with $x(z)=\max\{x(v)|~v\in W\}$ and $d_{W}(z)=q$. Note that $z\in W$. Then $d_{G-S}(z)\leq a-1$. By $\rho(G)x=A(G)x$, we have
\begin{eqnarray*}
 &\rho(G) x(z)&\leq \sum_{1\leq i\leq a-1-q}x(u_i)+q x(z)+\sum_{v\in S}x(v),\\
 &\rho(G) x(u_{n-s-t})&\geq  \sum_{1\leq i\leq a-1-q}x(u_i)+(n\!-\!s\!-\!t\!-\!a\!+\!q)x(u_{n-s-t})+\sum_{v\in S}x(v),
\end{eqnarray*}
from which we get that
$$(\rho(G)-q)(x(u_{n-s-t})-x(z))\geq (n-s-t-a)x(u_{n-s-t})>0$$
due to $n\geq t(b+2)+1$, $t\geq b+1$ and $s\leq t-1$. Combining this with (\ref{equ::lower}) and $q\leq a-1$, we have $\rho(G)>n-b-1>a-1\geq q$. It follows that $x(u_{n-s-t})>x(z)$. Combining this with  $x(w)\geq x(u_{n-s-t})$ and $x(z)\geq x(u)$, we get $x(w)>x(u)$. Let $G'=G-uv+wv$. Then $\sum_{v\in W}d_{G'-S}(v)=\sum_{v\in W}d_{G-S}(v)-1<at-bs-1$. By Lemmas \ref{lem::2.0} and \ref{lem::2.3}, we obtain that $G'$ contains no $[a,b]$-factors and $\rho(G')>\rho(G)$, which contradicts the maximality of $\rho(G)$. This implies that $e(W)=0$.
 %We next assert that $s\geq a$. Otherwise, $s\leq a-1$. Combining this with  (\ref{equ::upper-edge}), we get
%		\begin{equation*}
%			\begin{aligned}
%				\rho(G) & \leq  \frac{a-1}{2}+\sqrt{2e(G)-n a+\frac{(a+1)^2}{4}} \\
%				& \leq \sqrt{\Big(n\!-\!b\!-\!\frac{a}{2}\!-\!\frac{3}{2}\Big)^2\!-\!(2(t\!-\!b\!-\!1)n\!+\! a(b\!-\!2t\!+\!1)\!+\!b(b\!+\!2s\!+\!3)\!-\!t(2s\!+\!t\!+\!1)\!+\!4)}\\
%&~~+\frac{a\!-\!1}{2}\\
%				&\leq\sqrt{\Big(n\!-\!b\!-\!\frac{a}{2}\!-\!\frac{3}{2}\Big)^2\!-\!
%((2b\!+\!3)t^2\!-\!(2b^2\!+\!2a\!+\!6b\!+\!2s\!+\!3)t\!+\!ab\!+\!b^2\!+\!2sb\!+\!a\!+\!b\!+\!2)}\\
%	&~~+\frac{a\!-\!1}{2}~~~(\mbox{since $n\geq t(b+2)+1$}) \\			
%& \leq  \frac{a\!-\!1}{2}\!+\!\sqrt{\Big(n\!-\!b\!-\!\frac{a}{2}\!-\!\frac{3}{2}\Big)^2\!-\!((2b\!+\!3)t^2\!-\! (2b^2\!+\!4a\!+\!6b\!+\!1)t\!+\!3ab\!+\!b^2\!+\!a\!-\!b\!+\!2)}\\
%&~~(\mbox{since $s\leq a-1$}) \\
%& \leq  \frac{a\!-\!1}{2}\!+\!\sqrt{\Big(n\!-\!b\!-\!\frac{a}{2}\!-\!\frac{3}{2}\Big)^2\!-\!(2b^2-ab-7a+6b + 12)}~~(\mbox{since $t\geq b+2$}) \\
%&<n-b-2 ~~(\mbox{since $b>a\geq 1$}),
%			\end{aligned}
%		\end{equation*}
%which contradicts (\ref{equ::lower}). It follows that $s\geq a$.

{\flushleft\bf{Subcase 2.1}} $t=b+1$.

If  $s\geq a+1$, then
$$0\leq \sum_{v\in W}d_{G-S}(v)\leq at-bs-1=(a-s)b+a-1\leq a-b-1<0$$
due to $b>a$ and $t=b+1$, a contradiction. If $s\leq a-1$, since $d_{G}(v)\geq \delta(G)\geq a$, we have $d_{G-S}(v)=d_{U}(v)\geq a-s$ for $v\in W$. Recall that $U=\{u_1,u_2,\ldots,u_{n-s-t}\}$ with $x(u_1)\geq x(u_2)\geq \cdots\geq x(u_{n-s-b-1})$. By the maximality of $\rho(G)$ and Lemma \ref{lem::2.0}, we have $\{u_1,u_2,\dots,u_{a-s}\}\subseteq N_{G}(v)$ for any $v\in W$. Let $S'=S\cup \{u_1,u_2,\ldots,u_{a-s}\}$. Then $|S'|=a$ and
$$\sum_{v\in W}d_{G-S'}(v)=\sum_{v\in W}d_{G-S}(v)\!-\!(a\!-\!s)(b\!+\!1)\leq at-ab+s-a-1<at-b|S'|-1$$
due to $s\leq a-1$ and (\ref{equ::thm1-1}), which contradicts the maximality of $s$. This implies that $s=a$. Combining this with $t=b+1$, we have $G\in \mathcal{G}_{n}^{a,b}$. By Lemma \ref{lem::maximum}, we get
$G\cong H_{n}^{a,b}$, as required.

{\flushleft\bf{Subcase 2.2}} $t\geq b+2$.

 If $s\leq a-1$, since $n\geq t(b+2)+1$ and $t\geq b+2$, we can deduce a contradiction by using a similar analysis as Case 1. Thus, we consider $s\geq a$. Let $W=W_1\cup W_2$ with $W_1=\{w_1,w_2,\ldots,w_{t-b-1}\}$ and $W_2=\{w_{t-b},w_{t-b+1},\ldots,w_{t}\}$ such that $x(w_1)\geq x(w_2)\geq\cdots \geq x(w_t)$, and let $S=S_1\cup S_2$ with $S_{1}=\{v_1,v_2,\ldots,v_{s-a}\}$ and $S_2=\{v_{s-a+1},\ldots,v_{s}\}$. For $1\leq i\leq n-s-t$ and $1\leq j\leq t$, we have $N_{G}(w_j)\backslash\{u_{i}\}\subseteq N_{G}(u_{i})\backslash\{w_j\}$, and hence $x(u_{i})>x(w_j)$ by Lemma \ref{lem::adjacent}. From $A(G)x=\rho(G)x$, we have
\begin{eqnarray*}
 &\rho(G) x(u_{n-s-t})&\geq \sum_{1\leq i\leq n-s-t-1}x(u_i)+\sum_{1\leq i\leq s}x(v_i),\\
 &\rho(G) x(v_1)&= \sum_{1\leq i\leq n-s-t}x(u_i)+\sum_{2\leq i\leq s}x(v_i)+\sum_{1\leq i\leq t}x(w_i).\\
\end{eqnarray*}
Thus,
\begin{equation*}
\begin{aligned}
(\rho(G)+1)(2x(u_{n-s-t})-x(v_1))\geq&x(u_{n-s-t})+\sum_{1\leq i\leq n-s-t-1}x(u_i)+\sum_{1\leq i\leq s}x(v_i)-\sum_{1\leq i\leq t}x(w_i)\\
>&x(u_{n-s-t})+\sum_{1\leq i\leq n-s-2t-1}x(u_i)\\
>&0
\end{aligned}
\end{equation*}
due to $n\geq t(b+2)+1$, $s\leq t-1$ and $x(u_i)>x(w_j)$ for $1\leq i\leq n-s-t$ and $1\leq j\leq t$. It follows that $2x(u_{n-s-t})>x(v_1)$.

 Suppose that $E_1=\{uv\in E(G)|~u\in S_1\cup U, v\in W_2\}$ and $E_2=\{uv|~u\in U, v\in W_1\}\cup \{w_iw_j|~1\leq i<j\leq t-b-1\}$. Let $G^*=G-E_1+E_2$ and let $y$ be the Perron vector of $A(G^*)$. Clearly, $G^*\cong K_{a}\vee (K_{n-a-b-1}\cup (b+1)K_1)$. By symmetry, we have $y(v)=y(w_1)$ for $v\in V(G)\backslash(W_2\cup S_2)$, $y(v)=y(v_s)$ for $v\in S_2$ and $y(v)=y(w_{t-b})$ for $v\in W_2$. From $A(G^*)y=\rho(G^*)y$, we have
\begin{align*}
\rho(G^*)y(w_{t-b}) &=ay(v_s),\\
\rho(G^*)y(w_1) &=(n-a-b-2)y(w_1)+ay(v_s),\\
\rho(G^*)y(v_s) &=(n-a-b-1)y(w_1)+(a-1)y(v_s)+(b+1)y(w_{t-b}).
\end{align*}
Since $N_{G}(w_{t-b})\subseteq N_{G}(w_1)$, we have $y(w_1)>y(w_{t-b})$ by Lemma \ref{lem::adjacent}. By using a simple calculation, we get
\begin{equation*}
\begin{aligned}
 \rho(G^*)(2y(w_1)-y(v_s))=&(n-a-b-3)y(w_1)+(a+1)y(v_s)-(b+1)y(w_{t-b})\\
 >&(n-a-2b-4)y(w_1)+(a+1)y(v_s)\\
 >&0,
\end{aligned}
\end{equation*}
and hence $2y(w_1)>y(v_s)$. Combining this with $n\geq t(b+2)+1$, $t\geq b+2$ and $b>a$, we get
\begin{eqnarray*}
 \rho(G^*)(y(w_1)-2y(w_{t-b}))=(n\!-\!a\!-\!b\!-\!2)y(w_1)-ay(v_s)>(n-3a-b-2)y(w_1)>0.
\end{eqnarray*}
This implies that $y(w_1)>2y(w_{t-b})$. Assume that $e(U,W_i)=l_i$ for $i=1,2$. Thus, $\sum_{v\in W}d_{G-S}(v)=l_1+l_2\leq at-bs-1$. Combining this with $n\geq t(b+2)+1$, $t\geq b+2$ and $b>a$, we get
\begin{equation}\label{equ::big}
\begin{aligned}
&(n\!-\!s\!-\!t)(t\!-\!b\!-\!1)\!-\!(l_1\!+\!l_2)\!-\!(s\!-\!a)(b\!+\!1)\\
\geq & n\!-\!s-t-(at-bs-1)-(s-a)(b+1)~~~(\mbox{since $t\geq b+2$ and $l_1+l_2\leq at-bs-1$})\\
\geq & t(b-a)+t-2s+ab+a+2~~~(\mbox{since $n\geq t(b+2)+1$})\\
>& 2(t-s)+ab+a+2~~~(\mbox{since $b>a$})\\
\geq& ab+a+4 ~~~(\mbox{since $s\leq t-1$})\\
>&0.
\end{aligned}
\end{equation}
  Note that $y(w_1)>2y(w_{t-b})$ and $2x(u_{n-s-t})>x(v_1)$. Then
  \begin{equation}\label{equ::big-1}
\begin{aligned}
x(u_{n-s-t})y(w_1)-x(v_1)y(w_{t-b})>y(w_{t-b})(2x(u_{n-s-t})-x(v_1))>0.
\end{aligned}
\end{equation}
Since $N_{G}(u_1)\backslash \{v_1\}\subseteq N_{G}(v_1)\backslash \{u_1\}$, we have $x(v_1)>x(u_1)$ by Lemma \ref{lem::adjacent}. Therefore,
\begin{equation*}
\begin{aligned}
&~~~y^{T}(\rho(G^*)-\rho(G))x\\
&=y^{T}(A(G^*)-A(G))x\\
   &=\sum_{u_iv_j\in E_2}(x(u_{i})y(v_{j})\!+\!x(v_{j})y(u_{i}))\!-\!\sum_{u_iv_j\in E_1}\!(x(u_{i})y(v_{j})\!+\!x(v_{j})y(u_{i}))\\
   &\geq((n\!-\!s\!-\!t)(t\!-\!b\!-\!1)\!-\!l_1)(x(u_{n-s-t})y(w_1)\!+\!x(w_{t-b-1})y(u_1))\!+\!(t\!-\!b\!-\!1)(t\!-\!b\!-\!2)
   x(w_{t-b-1})y(w_{1})\\
   \end{aligned}
\end{equation*}
\begin{equation*}
\begin{aligned}
   &~~~\!-\!(s\!-\!a)(b\!+\!1)(x(v_1)y(w_{t-b})\!+\!x(w_{t-b})y(v_1))\!-\!l_2(x(u_1)y(w_{t-b})\!+\!x(w_{t-b})y(u_1))\\
   &> ((n\!-\!s\!-\!t)(t\!-\!b-1)\!-\!l_1)(x(u_{n-s-t})y(w_1)\!+\!x(w_{t-b-1})y(v_1))\!-\!
   ((s-a)(b+1)+l_2)\cdot\\
   &~~(x(v_1)y(w_{t-b})+x(w_{t-b})y(v_1))~~~~(\mbox{since $x(v_1)>x(u_1)$ and $y(v_1)=y(u_1)$})\\
    &>((n\!-\!s\!-\!t)(t\!-\!b\!-\!1)\!-\!(s\!-\!a)(b\!+\!1)\!-\!(l_1\!+\!l_2))x(w_{t-b-1})y(v_1)\\
    &~~(\mbox{by (\ref{equ::big}), (\ref{equ::big-1}) and $x(w_{t-b-1})\geq x(w_{t-b})$})\\
    &>0.
\end{aligned}
\end{equation*}
It follows that $\rho(G^*)>\rho(G)$. Note that $0=\sum_{v\in W_2}d_{G^*-S_2}(v)\leq a|W_2|-b|S_2|-1=a(b+1)-ab-1=a-1$. Then $G^*$ is a graph contains no $[a,b]$-factors and has a larger spectral radius than $G$, a contradiction.

This completes the proof.\end{proof}

\section{Concluding remarks}
%Let $h: E(G) \rightarrow[0,1]$ be a function on $E(G)$ and let $a, b$ be two positive integers with $a\leq b$. If $a\leq \sum_{e \in E_G(v)} h(e)\leq b$ for any $v \in V(G)$, then the spanning subgraph with edge set $E_h=\{e \in E(G) \mid h(e)>0\}$, denoted by $G\left[E_h\right]$, is called a \textit{fractional $[a, b]$-factor}
%of $G$ with indicator function $h$.
%
%Li and Fan \cite{LFZ25} gave a lower bound of spectral radius condition to guarantee the existence of a fractional $k$-factor in a graph.
%\begin{prob}
%Let $k\geq 3$ and let $G$ be a connected graph of order $n$ with minimum degree $\delta \geq k$ containing no fractional $k$-factors. Then
%$$\rho(G)\leq\rho(H^{k,k}_{n}),$$
%with equality if and only if $G\cong H^{k,k}_{n}$.
%\end{prob}
%\emph{}
%They gave the answer for $k=2$. Note that a graph contains $(g,f)$-factor. Then it also contains a fractional $(g,f)$-factor. Theorem \ref{thm::1.1} completely solves this problem.

Employing a proof technique analogous to that of Theorem 1.4 in \cite{LFZ25}, we obtain the following result, which addresses Problem 1 from the perspective of size conditions.
\begin{thm}\label{thm::edge}
Let $a$ and $b$ be two positive integers with $b>a$, and let $G$ be a connected graph of order $n\geq 4a+\frac{5b}{2}+6$ with minimum degree $\delta(G)\geq a$. If
		$$e(G) \geq {n-b-1\choose 2}+ab+2a,$$
then $G$ contains an $[a,b]$-factor.	
	\end{thm}
The condition in Theorem~\ref{thm::edge} is tight. For $G\in\mathcal{G}_{n}^{a,b}$, we have $e(G)={n-b-1\choose 2}+ab+2a-1$, but $G$ contains no fractional $[a,b]$-factors.

Let $h: E(G) \rightarrow[0,1]$ be a function on $E(G)$ and let $a, b$ be two positive integers with $a\leq b$. If $a\leq \sum_{e \in E_G(v)} h(e)\leq b$ for any $v \in V(G)$, then the spanning subgraph with edge set $E_h=\{e \in E(G) \mid h(e)>0\}$, denoted by $G\left[E_h\right]$, is called a \textit{fractional $[a, b]$-factor} of $G$ with indicator function $h$. Note that if $G$ contains an $[a,b]$-factor, then it also contains a fractional $[a,b]$-factor. As a consequence of Theorem \ref{thm::1.1}, we have the following result.

\begin{thm}
Let $a$ and $b$ be two positive integers with $b>a$, and let $G$ be a connected graph of order $n \geq 2(a+b+2)(b+2)$ with minimum degree $\delta(G)\geq a$. If
		$$\rho(G) \geq \rho(H_{n}^{a,b}),$$
then $G$ contains a fractional $[a,b]$-factor, unless $G\cong H_{n}^{a,b}$.
	\end{thm}

	\section*{Declaration of competing interest}
	The authors reported no potential competing interest.

\end{document}